\documentclass[11pt,reqno]{amsart}
\usepackage{color}

\usepackage{amsmath}
\usepackage{amsfonts,amscd}
\usepackage{amssymb}
\usepackage{url}
\usepackage{hyperref}

\usepackage[english]{babel}
\usepackage{booktabs}
\usepackage{tikz}

\theoremstyle{plain}
\newtheorem{theorem}                 {Theorem}      [section]

\newtheorem{corollary}    [theorem]  {Corollary}

\newtheorem{proposition}  [theorem]  {Proposition}

\theoremstyle{definition}

\newtheorem{definition}   [theorem]  {Definition}

\newtheorem{example}      [theorem]  {Example}

\numberwithin{equation}{section}

\def \cn{{\mathbb C}}
\def \C{{\mathbb C}}
\def \hn{{\mathbb H}}

\def \rn{{\mathbb R}}

\def \B{\mathcal B}
\def \E{\mathcal E}
\def \F{\mathcal F}

\def \S{\mathcal S}

\def\nab#1#2{\hbox{$\nabla$\kern -.3em\lower 1.0 ex
		\hbox{$#1$}\kern -.1 em {$#2$}}}
\def\hatnab#1#2{\hbox{$\nabla$\kern -.3em\lower 1.0 ex
		\hbox{$#1$}\kern -.1 em {$#2$}}}

\def \Re{\mathfrak R\mathfrak e}

\def \g{\mathfrak{g}}

\def \m{\mathfrak{m}}

\def \p{\mathfrak{p}}

\def \GLR#1{\mathbf{GL}_{#1}(\rn)}

\def \glr#1{\mathfrak{gl}_{#1}(\rn)}
\def \GLC#1{\mathbf{GL}_{#1}(\cn)}
\def \glc#1{\mathfrak{gl}_{#1}(\cn)}
\def \GLH#1{\mathbf{GL}_{#1}(\hn)}

\def \SL2{\widetilde{\text{\bf SL}}_{2}(\rn)}

\def \SO#1{\mathbf{SO}(#1)}
\def \so#1{\mathfrak{so}(#1)}

\def \U#1{\text{\bf U}(#1)}
\def \u#1{\mathfrak{u}(#1)}

\def \SU#1{\text{\bf SU}(#1)}
\def \su#1{\mathfrak{su}(#1)}

\def \Sp#1{\text{\bf Sp}(#1)}
\def \sp#1{\mathfrak{sp}(#1)}

\DeclareMathOperator{\Div}{div} 

\DeclareMathOperator{\trace}{trace}

\numberwithin{equation}{section}
\allowdisplaybreaks

\begin{document}

	

\author{Johanna Marie Gegenfurtner}
\address{Mathematics, Faculty of Science\\
Lund University\\
Box 118, Lund 221 00\\
Sweden}
\email{Johanna.Gegenfurtner@gmail.com}

\author{Sigmundur Gudmundsson}
\address{Mathematics, Faculty of Science\\
Lund University\\
Box 118, Lund 221 00\\
Sweden}
\email{Sigmundur.Gudmundsson@math.lu.se}

\title
[Compact Minimal Submanifolds of Symmetric Spaces]
{Compact Minimal Submanifolds of\\ 
the Riemannian Symmetric Spaces\\ 
$\SU n/\SO n$, $\Sp{n}/\U n$, $\SO{2n}/\U n$, $\SU{2n}/\Sp n$\\ 
via Complex-Valued Eigenfunctions}

\begin{abstract}
In this work we construct new multi-dimensional families of compact minimal submanifolds of the classical Riemannian symmetric spaces $\SU n/\SO n$, $\Sp{n}/\U n$, $\SO{2n}/\U n$ and $\SU{2n}/\Sp n$ of codimension two.
\end{abstract}
	
%\dedicatory{version 2.008 - \today - current editor: SG}

\maketitle

%%%%%%%%%%%%%%%%%%%%%%%%%%%%%%%%%%%%%%%%%%%
\section{Introduction}
\label{section-introduction}
%%%%%%%%%%%%%%%%%%%%%%%%%%%%%%%%%%%%%%%%%%%

The study of minimal submanifolds of a given ambient space plays a central role in differential geometry.  This has a long, interesting history and has attracted the interests of profound mathematicians for many generations.  The famous Weierstrass-Enneper representation formula, for minimal surfaces in Euclidean $3$-space, brings {\it complex analysis} into play as a useful tool for the study of these beautiful objects.

This was later generalised to the study of minimal surfaces in much more general ambient manifolds via {\it harmonic conformal immersions}.  The next  result follows from the seminal paper \cite{Eel-Sam} of Eells and Sampson from 1964.  For this see also Proposition 3.5.1 of \cite{Bai-Woo-book}.

\begin{theorem}
Let $\phi:(M^m,g)\to (N,h)$ be a smooth conformal map between Riemannian manifolds.  If $m=2$ then $\phi$ is harmonic if and only if the image is minimal in $(N,h)$.
\end{theorem}

This result has turned out to be very useful in the construction of minimal surfaces in Riemannian symmetric spaces of various types.  For this we refer to \cite{Cal},
\cite{Eel-Woo}, \cite{Uhl}, \cite{Bur-Raw} and \cite{Bur-Gue}, just to name a few.
\smallskip

In their work \cite{Bai-Eel} from 1981, Baird and Eells have shown that complex-valued harmonic morphisms from Riemannian manifolds are useful tools for the study of minimal submanifolds of codimension two. 

\begin{theorem}\cite{Bai-Eel}
\label{theorem:Bai-Eel-special}
Let $\phi:(M,g)\to\cn$ be a complex-valued harmonic morphism from a Riemannian manifold.  Then every regular fibre of $\phi$ is a minimal submanifold of $(M,g)$ of codimension {\it two}.
\end{theorem}

This can be seen as dual to the above mentioned generalisation of the Weierstrass-Enneper representation. Harmonic morphisms are the much studied {\it horizontally conformal harmonic}  maps.  For an introduction to the theory we recommend the book \cite{Bai-Woo-book}, by Baird and Wood, and the regularly updated online bibliography \cite{Gud-bib}.

\section{The Main Results}

The recent work \cite{Gud-Mun-1} introduces a method for constructing minimal submanifolds of Riemannian manifolds  via {\it submersions}, see Theorem \ref{theorem-Gud-Mun-1}.  Then this scheme is employed to  provide compact examples in several important cases. The main ingredients for this new procedure are the so called {\it complex-valued eigenfunctions} on the Riemannian ambient space.  These are functions which are eigen both with respect to the classical {\it Laplace-Beltrami} and the so called {\it conformality} operator, see Section \ref{section-eigenfunctions}.
\smallskip

In the current study we continue the investigation and apply the above mentioned method to the classical Riemannian symmetric spaces $$\SU n/\SO n,\ \Sp{n}/\U n,\ \SO{2n}/\U n\ \ \text{and}\ \ \SU{2n}/\Sp n.$$ 

In the first case we construct a complex $(n-1)$-dimen\-sional family 
$$\F_1=\{\phi_a^{-1}(\{0\})\ |\ [a]\in\cn P^{n-1}\}$$
of minimal submanifolds of $\SU n/\SO n$ of codimension two.

\begin{theorem}\label{theorem-SUSO}
For $n\ge 2$, let $a$ be a non-zero element of $\cn^n$ and $A\in\cn^{n\times n}$ with $A=a^t a$. Further let $\phi_a:\SU{n}/\SO{n}\to\cn$ be the complex-valued eigenfunction induced by its $\SO n$-invariant lift $\hat\phi_a:\SU n\to\cn$ given by
$$\hat\phi_a(z)=\trace(A z z^t).$$
Then the origin $0\in\cn$ is a regular value of $\phi_a$, so the compact fibre $\phi_a^{-1}(\{0\})$ is a minimal submanifold of $\SU{n}/\SO{n}$ of codimension two.
\end{theorem}

For the Riemannian symmetric space $\Sp{n}/\U{n}$ we construct a complex $(2n-1)$-dimensional family 
$$\F_2=\{\phi_a^{-1}(\{0\})\ |\ [a]\in\cn P^{2n-1}\}$$
of minimal submanifolds of codimension two. 

\begin{theorem}\label{theorem-SpU}
Let $a$ be a non-zero element of $\cn^{2n}$ and $A\in\cn^{2n\times 2n}$ with $A=a^t a$. Further let $\phi_a:\Sp{n}/\U{n}\to\cn$ be the complex-valued eigenfunction induced by its $\U n$-invariant lift $\hat\phi_a:\Sp n\to\cn$ given by
$$\hat\phi_a(q)=\trace(A qq^t).$$
Then the origin $0\in\cn$ is a regular value of $\phi_a$, so the compact fibre $\phi_a^{-1}(\{0\})$ is a minimal submanifold of $\Sp{n}/\U{n}$ of codimension two.
\end{theorem}

In the third case we construct a real $(6n-3)$-dimen\-sional family 
$$\F_3=\{\phi_{a,b}^{-1}(\{0\})\ |\ [a]\in\cn P^{2n-1},\ [b]\in\rn P^{2n-1},\ (a,a)=(a,b)=0\}$$
of minimal submanifolds of $\SO{2n}/\U{n}$ of codimension two.

\begin{theorem}\label{theorem-SOU}
For $n\ge 2$, let $a=u+iv\in\cn^{2n}$ and $b\in\rn^{2n}$ be two linearly independent elements, such that 
$(a,a)=(a,b)=0$. 
Let $A\in\cn^{2n\times 2n}$ be the skew-symmetric matrix $$A=\sum_{j,\alpha=1}^{2n} a_jb_\alpha Y_{j\alpha}.$$ 
Further let $\phi_{a,b}:\SO{2n}/\U{n}\to\cn$ be the complex-valued eigenfunction induced by its $\U n$-invariant lift $\hat\phi_{a,b}:\SO{2n}\to\cn$ given by
$$\hat\phi_{a,b}(x)=\trace(A xJ_n x^t).$$
Then the origin $0\in\cn$ is a regular value of $\phi_{a,b}$, so the compact fibre $\phi_{a,b}^{-1}(\{0\})$ is a minimal submanifold of $\SO{2n}/\U{n}$ of codimension two.
\end{theorem}

For the Riemannian symmetric space $\SU{2n}/\Sp{n}$ we construct a complex $(4n-2)$-dimensional family 
$$\F_4=\{\phi_{a,b}^{-1}(\{0\})\ |\ [a],[b]\in\cn P^{2n-1}, [a]\neq [b]\}$$
of minimal submanifolds of codimension two.

\begin{theorem}\label{theorem-SUSp}
For $n\geq2$, let $a,b$ be two linearly independent elements of $\cn^{2n}$ and $A\in\cn^{2n\times 2n}$ be the skew-symmetric matrix $$A=\sum_{j,\alpha=1}^{2n} a_jb_\alpha Y_{j\alpha}.$$     
Further let $\phi_{a,b}:\SU{2n}/\Sp{n}\to\cn$ be the complex-valued eigenfunction induced by its $\Sp n$-invariant lift $\hat\phi_{a,b}:\SU{2n}\to\cn$ given by
$$\hat\phi_{a,b}(z)=\trace(A z J_n z^t).$$
Then the origin $0\in\cn$ is a regular value of $\phi_{a,b}$  and the compact fibre $\phi_{a,b}^{-1}(\{0\})$ is a minimal submanifold of $\SU{2n}/\Sp{n}.$
\end{theorem}

The proofs of these four results are provided below. 
The readers interested in further details are referred to \cite{Geg-MSc}.

%%%%%%%%%%%%%%%%%%%%%%%%%%%%%%%%%%%%%%%%%%%
\section{Eigenfunctions and Eigenfamilies}
\label{section-eigenfunctions}
%%%%%%%%%%%%%%%%%%%%%%%%%%%%%%%%%%%%%%%%%%%

Let $(M,g)$ be an $m$-dimensional Riemannian manifold and $T^{\cn}M$ be the complexification of the tangent bundle $TM$ of $M$. We extend the metric $g$ to a complex bilinear form on $T^{\cn}M$.  Then the gradient $\nabla\phi$ of a complex-valued function $\phi:(M,g)\to\cn$ is a section of $T^{\cn}M$.  In this situation, the well-known complex linear {\it Laplace-Beltrami operator} (alt. {\it tension field}) $\tau$ on $(M,g)$.  In local coordinates this satisfies 
$$
\tau(\phi)=\Div (\nabla \phi)=\sum_{i,j=1}^m\frac{1}{\sqrt{|g|}} \frac{\partial}{\partial x_j}
\left(g^{ij}\, \sqrt{|g|}\, \frac{\partial \phi}{\partial x_i}\right).
$$
For two complex-valued functions $\phi,\psi:(M,g)\to\cn$ we have the following well-known fundamental relation
\begin{equation*}\label{equation-basic}
\tau(\phi\, \psi)=\tau(\phi)\,\psi +2\,\kappa(\phi,\psi)+\phi\,\tau(\psi),
\end{equation*}
where the complex bilinear {\it conformality operator} $\kappa$ is given by $$\kappa(\phi,\psi)=g(\nabla \phi,\nabla \psi).$$  Locally this satisfies 
$$\kappa(\phi,\psi)=\sum_{i,j=1}^mg^{ij}\cdot\frac{\partial\phi}{\partial x_i}\frac{\partial \psi}{\partial x_j}.$$

\begin{definition}\cite{Gud-Sak-1}\label{definition-eigenfamily}
Let $(M,g)$ be a Riemannian manifold. Then a complex-valued function $\phi:M\to\cn$ is said to be a {\it $(\lambda,\mu)$-eigenfunction} if it is eigen both with respect to the Laplace-Beltrami operator $\tau$ and the conformality operator $\kappa$ i.e. there exist complex numbers $\lambda,\mu\in\cn$ such that $$\tau(\phi)=\lambda\cdot\phi\ \ \text{and}\ \ \kappa(\phi,\phi)=\mu\cdot \phi^2.$$	
A set $\E =\{\phi_i:M\to\cn\ |\ i\in I\}$ of complex-valued functions is said to be a {\it $(\lambda,\mu)$-eigenfamily} on $M$ if there exist complex numbers $\lambda,\mu\in\cn$ such that for all $\phi,\psi\in\E$ we have 
$$\tau(\phi)=\lambda\cdot\phi\ \ \text{and}\ \ \kappa(\phi,\psi)=\mu\cdot \phi\,\psi.$$ 
\end{definition}
\medskip

For the standard odd-dimensional round spheres we have the following eigenfamilies based on the classical real-valued spherical harmonics.

\begin{example}\cite{Gud-Mun-1}\label{example-basic-sphere} 
Let $S^{2n-1}$ be the odd-dimensional unit sphere in the standard Euclidean space $\cn^{n}\cong\rn^{2n}$ and define $\phi_1,\dots,\phi_n:S^{2n-1}\to\cn$ by
$$\phi_j:(z_1,\dots,z_{n})\mapsto \frac{z_j}{\sqrt{|z_1|^2+\cdots +|z_n|^2}}.$$  Then the tension field $\tau$ and the conformality operator $\kappa$ on $S^{2n-1}$ satisfy
$$\tau(\phi_j)=-\,(2n-1)\cdot\phi_j\ \ \text{and}\ \ \kappa(\phi_j,\phi_k)=-\,1\cdot \phi_j\cdot\phi_k.$$
\end{example}
\medskip

For the standard complex projective space $\cn P^n$ we have a complex multi-dimensional eigenfamily, in a similar way.

\begin{example}\cite{Gud-Mun-1}\label{example-basic-projective-space}
Let $\cn P^n$ be the standard $n$-dimensional complex projective space. For a fixed integer $1\le\alpha < n+1$ and some $1\le j\le\alpha < k\le n+1$  define the function $\phi_{jk}:\cn P^n\to\cn$ by
$$\phi_{jk}:[z_1,\dots,z_{n+1}]\mapsto \frac
{z_j\cdot\bar z_k}{z_1\cdot \bar z_1+\cdots + z_{n+1}\cdot \bar z_{n+1}}.$$  
Then the tension field $\tau$ and the conformality operator $\kappa$ on $\cn P^n$ satisfy
$$\tau(\phi_{jk})=-\,4(n+1)\cdot\phi_{jk}\ \ \text{and}\ \ \kappa(\phi_{jk},\phi_{lm})=-\,4\cdot \phi_{jk}\cdot\phi_{lm}.$$
\end{example}

In recent years, explicit eigenfamilies of complex-valued functions have been found on all the classical compact Riemannian symmetric spaces. For this see Table \ref{table-eigenfamilies}.

\renewcommand{\arraystretch}{2}
\begin{table}[h]\label{table-eigenfamilies}
	\makebox[\textwidth][c]{
		\begin{tabular}{cccc}
			\midrule
			\midrule
$G/K$	& $\lambda$ & $\mu$ & Eigenfunctions \\
\midrule
\midrule
$\SO n$ & $-\,\frac{(n-1)}2$ & $-\,\frac 12$ & 
see \cite{Gud-Sak-1}\\
\midrule
$\SU n$ & $-\,\frac{n^2-1}n$ & $-\,\frac{n-1}n$ & 
see \cite{Gud-Sob-1} \\
\midrule
$\Sp n$ & $-\,\frac{2n+1}2$ & $-\,\frac 12$ & 
see \cite{Gud-Sak-1} \\
\midrule
$\SU n/\SO n$ & $-\,\frac{2(n^2+n-2)}{n}$& $-\,\frac{4(n-1)}{n}$ & 
see \cite{Gud-Sif-Sob-2} \\
\midrule
$\Sp n/\U n$ & $-\,2(n+1)$ & $-\,2$ & 
see \cite{Gud-Sif-Sob-2} \\
\midrule
$\SO{2n}/\U n$ & $-\,2(n-1)$ & $-1$ & 
see \cite{Gud-Sif-Sob-2} \\
\midrule
$\SU{2n}/\Sp n$ & $-\,\frac{2(2n^2-n-1)}{n}$ & $-\,\frac{2(n-1)}{n}$ & 
see \cite{Gud-Sif-Sob-2} \\
\midrule
$\SO{m+n}/\SO m\times\SO n$ & $-(m+n)$ & $-2$ & 
see \cite{Gha-Gud-4} \\
\midrule
$\U{m+n}/\U m\times\U n$ & $-2(m+n)$ & $-2$ & 
see \cite{Gha-Gud-5} \\
\midrule
$\Sp{m+n}/\Sp m\times\Sp n$ & $-2(m+n)$ & $-1$ & 
see \cite{Gha-Gud-5} \\
\midrule
\midrule
\end{tabular}	
}
\bigskip
\caption{Eigenfamilies on the classical compact irreducible Riemannian symmetric spaces.}
\label{table-eigenfamilies}	
\end{table}
\renewcommand{\arraystretch}{1}

We conclude this section with the following two results, particularly useful in the above mentioned situations of compact Riemannian symmetric spaces.

\begin{proposition}\label{proposition-lift}
Let $\pi:(\hat M,\hat g)\to (M,g)$ be a harmonic Riemannian submersion between Riemannian manifolds. Further let $\phi:(M,g)\to\C$ be a smooth function and $\hat\phi:(\hat M,\hat g)\to\C$ be the composition $\hat\phi=\phi\circ\pi$. Then the corresponding tension fields $\tau$ and conformality operators $\kappa$  satisfy
$$\tau(\hat\phi)=\tau(\phi)\circ\pi\ \ \text{and}\ \
\kappa(\hat\phi,\hat\psi) = \kappa(\phi,\psi)\circ\pi.$$
\end{proposition}

\begin{proof} The arguments needed here can be found in \cite{Gud-Sif-Sob-2}.
\end{proof}

In the sequel, we shall apply the following immediate consequence of Proposition \ref{proposition-lift}.

\begin{corollary}\label{corollary-lift-eigen}
Let $\pi:(\hat M,\hat g)\to (M,g)$ be a harmonic Riemannian submersion.  For a complex-valued function smooth function $\phi:(M,g)\to\C$ let $\hat \phi:(\hat M,\hat g)\to\C$ be the composition $\hat \phi=\phi\circ\pi$. Then the following statements are equivalent
\begin{enumerate}
\item[(i)] 
$\phi:M\to\C$ is a $(\lambda,\mu)$-eigenfunction on $M$,
\item[(ii)] 
$\hat\phi:\hat M\to\C$ is a $(\lambda,\mu)$-eigenfunction on $\hat M$.
\end{enumerate}
\end{corollary}

%%%%%%%%%%%%%%%%%%%%%%%%%%%%%%%%%%%%%%%%%%%
\section{Minimal Submanifolds via Eigenfunctions}
\label{section-minimal-submanifolds}
%%%%%%%%%%%%%%%%%%%%%%%%%%%%%%%%%%%%%%%%%%%

The recent paper \cite{Gud-Mun-1} provides a new application of complex-valued eigenfunctions.  This is a method for constructing minimal submanifolds of codimension two.

\begin{theorem}
{\rm \cite{Gud-Mun-1}}\label{theorem-Gud-Mun-1}
Let $\phi:(M,g)\to\cn$ be a complex-valued eigenfunction on a Riemannian manifold, such that $0\in\phi(M)$ is a regular value for $\phi$.  Then the fibre $\phi^{-1}(\{0\})$ is a minimal submanifold of $M$ of codimension two.
\end{theorem}

The main aim of our work is to apply Theorem \ref{theorem-Gud-Mun-1} in several of the interesting cases when the Riemannian manifold $(M,g)$ is one of the classical compact symmetric spaces.

The next result, from Riedler and Siffert's paper \cite{Rie-Sif} supplies us with a straightforward way of checking whether an eigenfunction, on a compact and connected Riemannian manifold, attains the required value $0\in\cn$.

\begin{theorem}{\rm \cite{Rie-Sif}}\label{phi0}
Let $(M,g)$ be a compact and connected Riemannian manifold and let $\phi:M\rightarrow\cn$ be a $(\lambda,\mu)$-eigenfunction not identically zero. Then the following are equivalent.
\begin{enumerate}
\item $\lambda=\mu$.
\item $\rvert\phi\rvert^2$ is constant.
\item $\phi(x)\neq0$ for all $x\in M.$
\end{enumerate}
\end{theorem}

As an obvious consequence we have the following.

\begin{corollary}\label{corollary-phi0}
If $\phi:M\rightarrow\cn$ is a complex-valued $(\lambda,\mu)$-eigenfunction on a compact and connected Riemannian manifold $(M,g)$ such that $\lambda\neq\mu,$ then there exists $x\in M$ such that $\phi(x)=0$.
\end{corollary}

%%%%%%%%%%%%%%%%%%%%%%%%%%%%%%%%%%%%%%%%%%%%%%%%%%%%%%%
\section{The General Linear Group $\GLC n$}
\label{section-GLC}
%%%%%%%%%%%%%%%%%%%%%%%%%%%%%%%%%%%%%%%%%%%%%%%%%%%%%%%

In this section we now turn our attention to the concrete Riemannian matrix Lie groups embedded as subgroups of the complex general linear group.
\medskip

The group of linear automorphisms of $\cn^n$ is the complex general linear group $\GLC n=\{ z\in\cn^{n\times n}\,|\, \det z\neq 0\}$ of invertible $n\times n$ matrices with its standard representation 
$$z\mapsto
\begin{bmatrix}
	z_{11} & \cdots & z_{1n} \\
	\vdots & \ddots & \vdots \\
	z_{n1} & \cdots & z_{nn}
\end{bmatrix}.
$$
Its Lie algebra $\glc n$ of left-invariant vector fields on $\GLC n$ can be identified with $\cn^{n\times n}$ i.e. the complex linear space of $n\times n$ matrices.  We equip $\GLC n$ with its natural left-invariant Riemannian metric $g$ induced by the standard Euclidean inner product $\glc n\times\glc n\to\rn$ on its Lie algebra $\glc n$ satisfying
$$g(Z,W)\mapsto \Re\,\trace\, (Z\cdot \bar W^t).$$ 
For $1\le i,j\le n$, we shall by $E_{ij}$ denote the element of $\rn^{n\times n}$ satisfying
$$(E_{ij})_{kl}=\delta_{ik}\delta_{jl}$$ and by $D_t$ the diagonal
matrices $D_t=E_{tt}.$ For $1\le r<s\le n$, let $X_{rs}$ and
$Y_{rs}$ be the matrices satisfying
$$X_{rs}=\frac 1{\sqrt 2}(E_{rs}+E_{sr}),\ \ Y_{rs}=\frac
1{\sqrt 2}(E_{rs}-E_{sr}).$$
For the real vector space $\glc n$ we then have the canonical orthonormal basis $\B^\cn=\B\cup i\B$, where 
$$\B=\{Y_{rs}, X_{rs}\,|\, 1\le r<s\le n\}\cup\{D_{t}\,|\, t=1,2,\dots,n\}.$$
\vskip .1cm

Let $G$ be a classical Lie subgroup of $\GLC n$ with Lie algebra $\g$ inheriting the induced left-invariant Riemannian metric, which we shall also denote by $g$.  In the cases considered in this paper, $\B_{\g}=\B^\cn\cap\g$ will be an orthornormal basis for the subalgebra $\g$ of $\glc n$.  By employing the Koszul formula for the Levi-Civita connection $\nabla$ on $(G,g)$, we see that for all $Z,W\in\B_{\g}$ we have
\begin{eqnarray*}
	g(\nab ZZ,W)&=&g([W,Z],Z)\\
	&=&\Re\,\trace\, ((WZ-ZW)\bar Z^t)\\
	&=&\Re\,\trace\, (W(Z\bar Z^t-\bar Z^tZ))\\
	&=&0.
\end{eqnarray*}

If $Z\in\g$ is a left-invariant vector field on $G$ and $\phi:U\to\cn$ is a local complex-valued function on $G$ then the $k$-th order derivatives $Z^k(\phi)$ satisfy
\begin{equation*}\label{equation-diff-Z}
	Z^k(\phi)(p)=\frac {d^k}{ds^k}\bigl(\phi(p\cdot\exp(sZ))\bigr)\Big|_{s=0},
\end{equation*}
\smallskip 

\noindent
This implies that the tension field $\tau$ and the conformality operator $\kappa$ on $G$ fulfill 
\begin{equation*}\label{equation-tau}
	\tau(\phi)
	=\sum_{Z\in\B_\g}\bigl(Z^2(\phi)-\nab ZZ(\phi)\bigr)
	=\sum_{Z\in\B_\g}Z^2(\phi),
\end{equation*}	
\begin{equation*}\label{equation-kappa}
	\kappa(\phi,\psi)=\sum_{Z\in\B_\g}Z(\phi)\cdot Z(\psi),
\end{equation*}
where $\B_\g$ is the orthonormal basis $\B^\cn\cap\g$ for the Lie algebra $\g$.

%%%%%%%%%%%%%%%%%%%%%%%%%%%%%%%%%%%%%%%%%%%
\section{The Special Orthogonal Group $\SO n$}
%%%%%%%%%%%%%%%%%%%%%%%%%%%%%%%%%%%%%%%%%%%

In this section we introduce the basic tools, related to the special orthogonal group $\SO n$, needed for our later constructions.  The classical Lie group $\SO n$ is given by
$$\SO{n}=\{x\in\GLR{n}\ |\ x\cdot x^t=I_n,\ \det x =1\}.$$
For this we use its standard $n$-dimensional representation on $\cn^n$ with 
$$x\mapsto
\begin{bmatrix}
x_{11} & \cdots & x_{1n} \\ 
\vdots & \ddots & \vdots \\
x_{n1} & \cdots & x_{nn}
\end{bmatrix}.$$ 
The Lie algebra $\so n$ of $\SO n$ is the set of real  skew-symmetric matrices
$$\so n=\{X\in\glr n\ |\ X+X^t=0\}$$ 
and for this we have the canonical orthonormal basis
$$\B_{\so n}=\{Y_{rs}\ |\ 1\le r<s\le n\}.$$
The gradient $\nabla\phi$ of a complex-valued function $\phi:\SO n\to\cn$ is an element of the complexified tangent bundle $T^\cn\SO n$.  This satisfies 
$$\nabla\phi=\sum_{Y\in\B_{\so n}}Y(\phi)\cdot Y.$$
We now consider the functions $x_{j\alpha}: \SO{n} \to \cn$ with $x_{j\alpha} : x \mapsto e_j \cdot x \cdot e_\alpha^t$. For any tangent vector $Y\in B_{\u{n}}$ we have 
$$Y(x_{j\alpha}) : x \mapsto e_j \cdot x \cdot Y  \cdot e_\alpha^t = \sum_{k=1}^n x_{jk}\cdot Y_{k\alpha}.$$ 
This implies that 
\begin{equation}\label{equation-basic-SO(n)}
Y_{rs}(x_{j \alpha}) = \frac{1}{\sqrt{2}}\, (  
x_{jr}\cdot \delta_{\alpha s}
-x_{js}\cdot \delta_{\alpha r}).
\end{equation}

%%%%%%%%%%%%%%%%%%%%%%%%%%%%%%%%%%%%%%%%%%%%%%%%%%%%%%%
\section{The Unitary Groups $\U n$ and $\SU n$}
\label{section-unitary-group}
%%%%%%%%%%%%%%%%%%%%%%%%%%%%%%%%%%%%%%%%%%%%%%%%%%%%%%%

The unitary group $\U n$ is the compact subgroup of $\GLC n$ given by 
$$\U n=\{z\in\cn^{n\times n}\,|\,z\cdot\bar z^t=I\}.$$ 
For its standard complex representation $\U n\to\GLC n$ on $\cn^n$ we use the notation 
$$z\mapsto
\begin{bmatrix}
	z_{11} & \cdots & z_{1n} \\
	\vdots & \ddots & \vdots \\
	z_{n1} & \cdots & z_{nn}
\end{bmatrix}.
$$
The Lie algebra $\u n$ of $\U n$ consists of the skew-Hermitian matrices i.e. 
$$\u{n}=\{Z\in\cn^{n\times n}\ |\ Z+\bar Z^t=0\}.$$ 
The left-invariant metric on $\GLC n$ induces the standard biinvariant Riemannian metric $g$ on $\U n$, with
$$g(Z,W)=\Re\trace(Z\cdot\bar W^t).$$ 
The canonical orthonormal basis for the Lie algebra $\u n$ is then given by 
$$\B_{\u n}=\{Y_{rs}, iX_{rs}|\ 1\le r<s\le n\}\cup\{iD_t|\ t=1,\dots ,n\}.$$

The special unitary group $\SU n$ is the subgroup of $\U n$ satisfying 
$$\SU n=\{z\in\U n\ |\ \det(z)=1\}$$
with Lie algebra 
$$\su n=\{Z\in\u n\ |\ \trace Z=0\}.$$
For the Lie algebra $\u n$ of $\U n$ we have the orthogonal decomposition 
$$\u n=\su n\oplus\m,$$
where the one dimensional subspace $\m$ is generated by the unit vector $$E=i\,I_n/\sqrt n.$$

%%%%%%%%%%%%%%%%%%%%%%%%%%%%%%%%%%%%%%%%%%%
\section{The Quaternionic Unitary Group $\Sp n$}
%%%%%%%%%%%%%%%%%%%%%%%%%%%%%%%%%%%%%%%%%%%

The group $\Sp n$ is the intersection of the unitary group $\U{2n}$ and the standard representation of the quaternionic general linear group $\GLH n$ in $\cn^{2n\times 2n}$ given by
\begin{equation}\label{equation-Spn}
(z+jw)\mapsto q=
\begin{bmatrix}
z & w \\
 -\bar w & \bar z
\end{bmatrix}.
\end{equation}
The Lie algebra $\sp n$ of $\Sp n$ satisfies

$$\sp{n}=\left\{\begin{bmatrix} Z & W
\\ -\bar W & \bar Z\end{bmatrix}\in\cn^{2n\times 2n}
\ \Big|\ Z^*+Z=0,\ W^t-W=0\right\}.$$ 

We now  introduce the following notation for the elements of the orthonormal basis $\B_{\sp n}$ of the Lie algebra $\sp n$:
$$Y^a_{rs}=\frac 1{\sqrt 2}
\begin{bmatrix}
Y_{rs} & 0 \\
     0 & Y_{rs}
\end{bmatrix}, 
X^a_{rs}=\frac 1{\sqrt 2}
\begin{bmatrix}
iX_{rs} & 0 \\
      0 & -iX_{rs}
\end{bmatrix},$$
$$ X^b_{rs}=\frac 1{\sqrt 2}
\begin{bmatrix}
      0 & iX_{rs} \\
iX_{rs} & 0\end{bmatrix}, 
X^c_{rs}=\frac 1{\sqrt 2}
\begin{bmatrix}
      0 & X_{rs} \\
-X_{rs} & 0
\end{bmatrix},$$
$$D^a_{t}=\frac 1{\sqrt 2}
\begin{bmatrix}
iD_{t} & 0 \\
     0 & -iD_{t}
\end{bmatrix}, 
D^b_{t}=\frac 1{\sqrt 2}
\begin{bmatrix}
     0 & iD_{t}  \\
iD_{t} & 0
\end{bmatrix}, 
D^c_{t}=\frac 1{\sqrt 2}
\begin{bmatrix}
     0 & D_{t}  \\
-D_{t} & 0
\end{bmatrix}.$$
Here $1\le r<s\le n$ and $1\le t\le n$.
\smallskip

%%%%%%%%%%%%%%%%%%%%%%%%%%%%%%%%%%%%%%%%%%%
\section{The Symmetric Space $\SU n/\SO n$}
\label{section-SUn-SOn}
%%%%%%%%%%%%%%%%%%%%%%%%%%%%%%%%%%%%%%%%%%%

The purpose of this section is to prove Theorem \ref{theorem-SUSO} and thereby construct a new multi-dimensional family of compact minimal submanifolds of the homogeneous quotient manifold $\SU n/\SO n$, which carries the structure of a compact Riemannian symmetric space. It is well-known that the natural projection $\pi:\SU n\to\SU n/\SO n$ from the special unitary group $\SU n$ is a Riemannian submersion.  This means that we can apply Corollary \ref{corollary-lift-eigen} in this situation.
\smallskip

From the work \cite{Gud-Sif-Sob-2}, of Gudmundsson, Siffert and Sobak, we have the following construction of eigenfunctions.

\begin{proposition}
{\rm \cite{Gud-Sif-Sob-2}}\label{proposition-eigen-SUSO}
For $n\ge 2$, let $a$ be a non-zero element of $\cn^n$ and $A\in\cn^{n\times n}$ with $A=a^t a$. Further let $\phi_a:\SU{n}/\SO{n}\to\cn$ be the complex-valued function induced by its $\SO n$-invariant lift $\hat\phi_a:\SU n\to\cn$ given by
$$\hat\phi_a(z)=\trace(A z z^t).$$
Then $\phi_a$ is an eigenfunction on $\SU{n}/\SO{n}$ with eigenvalues $\lambda$ and $\mu$ satisfying $$\lambda=-\frac{2(n^2+n-2)}{n}\ \ \text{and}\ \  \mu=-\frac{4(n-1)}{n}.$$
\end{proposition}
\smallskip

\begin{proof}(Theorem \ref{theorem-SUSO})
For the Lie algebra $\su n$ of $\SU n$ we have the orthogonal decomposition $\su{n}=\so{n}\oplus i\mathfrak{p}$ where the subspace  $$\p=\{X\in\rn^{n\times n}\ | 
\  X=X^t \ \text{and}\ \trace X=0\}$$ 
is generated by the elements of 
$$\{X_{rs},\ D_{rs} \ | \ 1\leq r<s\leq n\}.$$
This means that if a point $z\in\SU n$ is a critical point for $\hat\phi:\SU n\to\cn$ then, for all $1\leq r<s\leq n$, we have
\begin{eqnarray*}
X_{rs}(\hat{\phi}_a)
&=&\trace(X_{rs}^t\cdot z^t\cdot A\cdot z+z^t\cdot A\cdot z\cdot X_{rs})\\
&=&\trace((X_{rs}^t+X_{rs})\cdot z^t\cdot a^t\cdot a\cdot z)\\
&=&2\,\trace(a\cdot z\cdot X_{rs}\cdot z^t\cdot a^t)\\
&=&2\sqrt 2\,(\sum_{k=1}^n a_k z_{kr})(\sum_{k=1}^n a_k z_{ks})\\
&=&0.
\end{eqnarray*}
Hence, for any choice $1\leq r<s\leq n$, either  \begin{equation}\label{Xrs-SUSO}
\sum_{k=1}^n a_k z_{kr}=0\ \ \text{or}\ \ \sum_{k=1}^n a_k z_{ks}=0.
\end{equation}

Now there exists an integer $1\le t\le n$ such that 
\begin{equation}\label{equation-SUSO-t}
\sum_{k=1}^n a_k z_{kt}\neq 0,
\end{equation}
otherwise the rows of the matrix $z\in\SU n$ would be linearly dependent, which is impossible.  It then follows from statement (\ref{Xrs-SUSO}) that if $m\neq t$, then 
$$\sum_{k=1}^n a_k z_{km}=0.$$

Similarly, we see that at a critial point $z\in\SU n$ we have 
$D_{rs}(\phi)=0$ for all $1\leq r<s\leq n$. This is equivalent to 
$$(\sum_{k=1}^n a_k z_{kr}+\sum_{k=1}^n a_k z_{ks})\cdot (\sum_{k=1}^n a_k z_{kr}-\sum_{k=1}^n a_k z_{ks})=0.$$ 
Applying our previous findings, we see that
$$\sum_{k=1}^n a_k z_{kt}=0,$$ which contradicts equation 
(\ref{equation-SUSO-t}).
We conclude that the eigenfunction $\phi_a:\SU n/\SO n\to\cn$ has no critical points. According to  Corollary \ref{corollary-phi0}, the origin $0\in\cn$ lies in the image of $\phi_a.$  It then immediately follows from Theorem \ref{theorem-Gud-Mun-1} that the fibre $\phi_a^{-1}(\{0\})$ is minimal.
\end{proof}

%%%%%%%%%%%%%%%%%%%%%%%%%%%%%%%%%%%%%%%%%%%
\section{The Symmetric Space $\Sp{n}/\U n$}
\label{section-Spn-Un}
%%%%%%%%%%%%%%%%%%%%%%%%%%%%%%%%%%%%%%%%%%%

In this section we prove Theorem \ref{theorem-SpU}, which yields a new multi-dimensional family of compact minimal submanifolds of the Riemannian symmetric space $\Sp{n}/\U n$. We first recall that the natural projection $\pi:\Sp n\to\Sp n/\U n$ from the quaternionic unitary group $\Sp n$ is a Riemannian submersion and thus we may now employ Corollary \ref{corollary-lift-eigen}.
\smallskip 

Here we use the standard representation of the quaternionic unitary group $\Sp n$ on $\cn^{2n}$ given by equation (\ref{equation-Spn}). Then we identify the unitary group with following subgroup of $\Sp n$
$$\U n\cong\Big\{\begin{bmatrix}
 x & y \\
-y & x
\end{bmatrix}
\in\rn^{2n\times 2n}\ \Big|\ xx^t+yy^t=I_n,\ \ xy^t=yx^t\Big\}.$$

For the Lie algebra $\sp n$ of $\Sp n$ we have the orthogonal decomposition $\sp{n}=\u{n}\oplus \p$ where the subspace $\p$ is generated by the elements of its following orthonormal basis
\begin{eqnarray*}
\mathcal{B}_\p&=&\Bigl\{ 
X_{rs}^{a}=\frac{i}{2}
\begin{bmatrix}
    X_{rs}&0\\
    0&-X_{rs}
\end{bmatrix}, 
D_{t}^{a}=\frac{i}{\sqrt{2}}
\begin{bmatrix}
    D_t&0\\
    0&-D_t
\end{bmatrix},\\
&&\qquad X_{rs}^{b}=\frac{i}{2}
\begin{bmatrix}
    0&X_{rs}\\
    X_{rs}&0
\end{bmatrix}, 
D_{t}^{b}=\frac{i}{\sqrt{2}}
\begin{bmatrix}
    0&D_t\\
    D_t&0
\end{bmatrix}\\
&&\qquad\qquad\qquad\qquad\qquad\Big| \ 1\leq r<s\leq n, \ 1\leq t\leq n\Bigr\}. 
\end{eqnarray*}

Let $\sigma:\Sp{n}\to\Sp{n}$ be the map given by $\sigma(q)=q\,q^t$.  Then it is easily seen that $\sigma$ is $\U{n}$-invariant i.e. for all $q\in\Sp{n}$ and $z\in\U{n}$ we have 
$$\sigma(qz)=qzz^tq^t=qq^t=\sigma(q).$$

The following result on eigenfunctions on the quotient space $\Sp{n}/\U{n}$ can be found in \cite{Gud-Sif-Sob-2}.

\begin{proposition}
{\rm \cite{Gud-Sif-Sob-2}}
\label{proposition-eigen-SpU}
Let $a$ be a non-zero element of $\cn^{2n}$ and $A\in\cn^{2n\times 2n}$ with $A=a^t a$. Further let $\phi_a:\Sp{n}/\U{n}\to\cn$ be the complex-valued function induced by its $\U n$-invariant lift $\hat\phi_a:\Sp n\to\cn$ given by
$$\hat\phi_a(q)=\trace(A q q^t).$$
Then $\phi_a$ is an eigenfunction on $\Sp{n}/\U{n}$ with eigenvalues
$$\lambda=-2(n+1)\ \ \text{and}\ \ \mu=-2.$$
\end{proposition}
\smallskip

\begin{proof}(Theorem \ref{theorem-SpU})
Let us assume that $q\in\Sp n$ is a critical point for the function $\hat\phi_a$ i.e. that the gradient $$\nabla\hat\phi_a(q)=\sum_{X\in\B_\p}X(\hat\phi_a(q))\cdot X=0.$$

By setting $D_t^{a}(\phi_a)=0$ and $D_t^{b}(\phi_a)=0$, for all $1\leq t\leq n$, we obtain the following relations
\begin{equation}\label{equation-SpU-1}
\left(\sum_{k=1}^{n} a_k q_{kt}\right)\cdot\left(\sum_{k=1}^{n} a_k q_{k,n+t}\right)=0
\end{equation}
and
\begin{equation}\label{equation-SpU-2}
\left(\sum_{k=1}^{n} a_k q_{kt}\right)^2-\left(\sum_{k=1}^{n} a_k q_{k,n+t}\right)^2=0,
\end{equation}
Equation (\ref{equation-SpU-1}) implies that for all $1\leq t\leq n$ either 
$$\sum_{k=1}^{n} a_k q_{kt}=0\ \ \text{or}\ \ \sum_{k=1}^{n} a_k q_{k,n+t}=0.$$ 
By plugging this into the equation (\ref{equation-SpU-1}) we then yield
$$\sum_{k=1}^{n} a_k q_{kt}=0 \ \ \text{and}\ \ \sum_{k=1}^{n} a_k q_{k,n+t}=0,$$
for all $1\leq t\leq n$.  But this is clearly impossible.  
\smallskip 

This shows that the eigenfunction $\hat\phi_a$ has no critical points and hence is a submersion. It follows from Corollary \ref{corollary-phi0} that the origin $0\in\cn $ lies in the image of $\phi_a$.  Theorem \ref{theorem-Gud-Mun-1} implies that the fibre $\phi_a^{-1}(\{0\})$ is a minimal submanifold of $\Sp n/\U n$ of codimension two.
\end{proof}

%%%%%%%%%%%%%%%%%%%%%%%%%%%%%%%%%%%%%%%%%%%
\section{The Symmetric Space $\SO{2n}/\U n$}
\label{section-SO2n-Un}
%%%%%%%%%%%%%%%%%%%%%%%%%%%%%%%%%%%%%%%%%%%

The goal of this section is to prove Theorem \ref{theorem-SOU}, which immediately provides us with a new multi-dimensional family of compact minimal submanifolds of the Riemannian symmetric space $\SO{2n}/\U n$. As the natural projection $\pi:\SO{2n}\to\SO{2n}/\U n$ from the special orthogonal group $\SO{2n}$ is a Riemannian submersion, we can now apply Corollary \ref{corollary-lift-eigen}.
\smallskip 

For the unitary group $\U n$ we have a well-known natural group monomorphism  
$$\psi:\U n\to\SO{2n},\ \ \text{with}\ \ \psi:z=x+iy\mapsto \begin{pmatrix}x&-y\\y&x\end{pmatrix},$$
embedding it into the special orthogonal group $\SO{2n}$. Here $x,y\in\rn^{n\times n}$ are the real and imaginary parts of $z\in\U n$.  In this section we identify $\U n$ with the image $\psi (\U n)$ which clearly is a subgroup of $\SO{2n}$. 

For the Lie algebra $\so{2n}$ we have the orthogonal decomposition $$\so{2n}=\u{n}\oplus\p,$$ where $\u{n}$ is the Lie algebra of $\U{n}$ as a subgroup of $\SO{2n}$.  A natural orthonormal basis for the vector space $\p$ is given by  
$$\mathcal{B}_\p=\left
\{Y_{rs}^{a}
=\frac{1}{\sqrt{2}}
\begin{pmatrix}
Y_{rs}&0\\
0&-Y_{rs}
\end{pmatrix}, 
Y_{rs}^{b}
=\frac{1}{\sqrt{2}}
\begin{pmatrix}
0&Y_{rs}\\
Y_{rs}&0
\end{pmatrix}
\ \Big| \ 1\leq r<s\leq n\right\}.$$

Let $I_n$ denote the $n \times n$ identity matrix and $J_n$ be the standard complex structure on $\rn^{2n}$ given by
\begin{equation*}
J_n = \begin{bmatrix}
0 & I_n\\
-I_n & 0
\end{bmatrix}.
\end{equation*}
Then consider the map $\sigma :\SO{2n}\to\SO{2n}$ satisfying 
\begin{equation*}
\sigma(x) = x\cdot J_n \cdot x^t.
\end{equation*}
A simple calculation shows that this map is $\U n$-invariant i.e. for any $y\in\U{n}$ and $x\in\SO{2n}$ we have $$\sigma (xy)=xyJ_ny^tx^t=xJ_nx^t=\sigma(x).$$
This shows that for any skew-symmetric matrix $A\in\cn^{2n\times 2n}$ the function $\hat\phi_A :\SO{2n}\to\cn$, given by
\begin{equation*}
\hat \phi_A(x) = \trace(A\sigma(x)) = \trace(x^tAxJ_n),
\end{equation*}
is $\U n$-invariant and hence induces a map $\phi_A:\SO{2n}/\U n\to\cn$ on the the quotient space. Note that in this case the skew-symmetry condition can be assumed without loss of generality, since $\sigma(x)$ is skew-symmetric.
\smallskip

We now have the following slightly improved version of Proposition 5.1. of \cite{Gud-Sif-Sob-2}.

\begin{proposition}\label{proposition-SOU}
For $n\ge 2$, let $a,b$ be two linearly independent elements of $\cn^{2n}$ satisfying the condition 
$$(a,a)(b,b)-(a,b)^2=0.$$ 
Let $A\in\cn^{2n\times2n}$ be the skew-symmetric matrix $$A=\sum_{j,\alpha=1}^{2n} a_jb_\alpha Y_{j\alpha}.$$ 
Further let $\phi_{a,b}:\SO{2n}/\U{n}\to\cn$ be the complex-valued function induced by its $\U n$-invariant lift $\hat\phi_{a,b}:\SO{2n}\to\cn$ given by
$$\hat\phi_{a,b}(x)=\trace(A x J_n x^t).$$
Then $\phi_{a,b}$ is an eigenfunction on $\SO{2n}/\U{n}$ with the eigenvalues
$$\lambda=-2(n-1)\ \ \text{and}\ \ \mu=-1.$$
\end{proposition}

\begin{proof}
The arguments needed here are exactly the same as those employed for  Proposition 5.1. in \cite{Gud-Sif-Sob-2}.
\end{proof}
\smallskip

\begin{proof}(Theorem \ref{theorem-SOU})
From the definition $$\hat\phi_{a,b}(x)=\trace(AxJ_nx^t)$$ of $\hat\phi_{a,b}$ it is easily seen that 
\begin{equation*}
\hat\phi_{a,b}(x)
=\frac{2}{\sqrt{2}}\cdot\sum_{t=1}^n\big((a,x_{n+t})(b,x_t)-(a,x_t)(b,x_{n+t})\big).
\end{equation*}
Since $x\in\SO{2n}$, we know that its columns $x_1,x_2,\dots,x_{2n}$ form an orthonormal basis for the vector space $\rn^{2n}$.  This gives the equation
\begin{equation*}\label{equation-ortho}
(a,b)=\sum_{k=1}^{2n}(a,x_k)(b,x_k)=0.
\end{equation*}
Corollary \ref{corollary-phi0} implies that there exists an element $x\in\SO{2n}$ satisfying $\hat\phi_{a,b}(x)=0.$
Let us now assume that the element $x\in\SO{2n}$ is in the fibre $\hat\phi_{a,b}(\{0\})$ i.e.
\begin{equation*}
\sum_{t=1}^n\big( (a,x_{n+t})(b,x_t)-(a,x_t)(b,x_{n+t})\big)=0.
\end{equation*}
The orthogonal complement $\p$ of the subspace $\u n$ in the Lie algebra $\so{2n}$ is generated by the elements
$$Y_{rs}^{a}=\begin{pmatrix}
    Y_{rs} & 0\\
    0 & -Y_{rs}
\end{pmatrix},
Y_{rs}^{b}=\begin{pmatrix}
    0&Y_{rs}\\
    Y_{rs}&0
\end{pmatrix}\in\mathcal{B}_{\p},$$
where $1\leq r<s\leq n$. 
A simple computation shows that for $Y_{rs}^{a},Y_{rs}^{b}\in\B_\p$ the conditions 
$$\Re\, Y_{rs}^{a}(\hat\phi_{a,b})=0\ \ \text{and}\ \  \Re\, Y_{rs}^{b}(\hat\phi_{a,b})=0$$ 
are equivalent to
\begin{equation*}
0=(u,x_r)(b,x_{n+s})+(u,x_{n+r})(b,x_s)-(u,x_s)(b,x_{n+r})-(u,x_{n+s})(b,x_r),
\end{equation*}
\begin{equation*}
0=(u,x_r)(b,x_{s})+(u,x_{n+s})(b,x_{n+r})-(u,x_s)(b,x_{r})-(u,x_{n+r})(b,x_{n+s}),
\end{equation*}
respectively. The non-vanishing vector $a=u+i\,v\in\cn^{2n}$ is isotropic i.e. $(a,a)=0$, hence $|u|=|v|\neq 0$.  This implies that there exists an $r\in\{1,2,\dots,2n\}$ such that $(u,x_r)\neq 0$.  Without loss of generality we can assume that $(u,x_1)\neq 0$. Then with $s=2,3,\dots n$ we yield the following system of $2n=2\,(n-1)+2$ equations
\begin{eqnarray*}
	0&=&(u,x_r)(b,x_{n+s})+(u,x_{n+r})(b,x_s)-(u,x_s)(b,x_{n+r})-(u,x_{n+s})(b,x_r),\\
	0&=&(u,x_r)(b,x_{s})+(u,x_{n+s})(b,x_{n+r})-(u,x_s)(b,x_{r})-(u,x_{n+r})(b,x_{n+s}),\\
	0&=&\Re\, (a,b)=\sum_{k=1}^{2n}(u,x_k)\,(b,x_k),\\
	0&=&\Re\, \hat\phi_{a,b}(x)=-\sum_{k=1}^{n}
	\big( (u,x_{n+k})\,(b,x_k)-(u,x_{k})\,(b,x_{n+k})\big).
\end{eqnarray*}
This can be written in the form
$$
M_{n}\cdot 
\begin{bmatrix}
(b,x_1)\\(b,x_2)\\ \vdots \\ (b,x_{2n})
\end{bmatrix}
=
\begin{bmatrix}
0\\0\\ \vdots \\ 0
\end{bmatrix}
.
$$
By reordering the columns of the matrix $M_{n}\in\rn^{2n\times2n}$ it can be brought on the form 
$$M_n \mapsto \tilde M_n = (u,x_1)\cdot I_{2n} + S_n,$$ 
where $I_{2n}$ is the identity matrix and $S_n\in\rn^{2n\times 2n}$ is skew-symmetric, see Appendix \ref{appendix-A}.  This means that the matrix $M_n$ is invertible.  This implies that the non-zero vector $b\in\rn^{2n}$ vanishes.  This contradicts the fact that an element $x\in\hat\phi_{a,b}^{-1}(\{ 0\})$ is a critical point of the $\U n$-invariant function $\hat\phi_{a,b}:\SO{2n}\to\cn$.  It then follows that the compact fibre $\phi_{a,b}^{-1}(\{ 0 \})$ is regular.  According to Theorem \ref{theorem-Gud-Mun-1} it is minimal as well.
\end{proof}

%%%%%%%%%%%%%%%%%%%%%%%%%%%%%%%%%%%%%%%%%%%
\section{The Symmetric Space $\SU{2n}/\Sp n$}
\label{section-SU2n-Spn}
%%%%%%%%%%%%%%%%%%%%%%%%%%%%%%%%%%%%%%%%%%%

In this last section we provide the proof of Theorem \ref{theorem-SUSp}. Consequently, we obtain a new multi-dimensional family of compact minimal submanifolds of the Riemannian symmetric space $\SU{2n}/\Sp n$. Similar to the previous sections, note that the natural projection $\pi:\SU{2n}\to\SU{2n}/\Sp n$ from the special unitary group $\SU{2n}$ is a Riemannian submersion and thus Corollary \ref{corollary-lift-eigen} applies.
\smallskip

For the quaternionic unitary group $\Sp n$ we have the well-known group monomorphism  
$$\psi:\Sp n\to\SU{2n},\ \ \text{with}\ \ \psi:q=z+jw\mapsto \begin{pmatrix}z&-\bar w\\w&\bar z\end{pmatrix},$$
embedding $\Sp n$ into the special orthogonal group $\SU{2n}$. Here $z,w\in\cn^{n\times n}$.  In this section we identify $\Sp n$ with the image $\psi (\Sp n)$ which clearly is a subgroup of $\SU{2n}$. 

For the Lie algebra $\su{2n}$ we have the orthogonal decomposition $$\su{2n}=\sp{n}\oplus\p,$$ where $\sp{n}$ is the Lie algebra of $\Sp{n}$ as a subgroup of $\SU{2n}$. Then the vector space $\p$ is generated by the elements of the set $$\S=\left\{Y_{rs}^{a}=\frac{1}{2}\begin{pmatrix}
    Y_{rs}&0\\
    0&-Y_{rs}
\end{pmatrix},X_{rs}^{a}=\frac{1}{2}\begin{pmatrix}
    iX_{rs}&0\\
    0&iX_{rs}
\end{pmatrix},Y_{rs}^{b}=\frac{1}{2}\begin{pmatrix}
    0&Y_{rs}\\
    Y_{rs}&0
\end{pmatrix},\right.$$ $$\left.Y_{rs}^{c}=\frac{1}{2}\begin{pmatrix}
    0&iY_{rs}\\
    -iY_{rs}&0
\end{pmatrix} \  \Big| \  1\leq r<s\leq n\right\}$$
$$\bigcup \ \left\{D_{rs}^{a}=\frac{1}{2}\begin{pmatrix}
    iD_{rs}&0\\
    0&iD_{rs}
\end{pmatrix} \ \Big| \ 1\leq r<s\leq n\right\},$$ where $D_{rs}=\frac{1}{\sqrt{2}}\cdot(E_{rr}-E_{ss}).$
\smallskip

Let us now consider the map $\sigma :\SU{2n}\to\SU{2n}$ satisfying 
\begin{equation*}
\sigma(x) = z\, J_n \, z^t.
\end{equation*}
A simple calculation shows that this map is $\Sp n$-invariant i.e. for any $q\in\Sp{n}$ and $z\in\SU{2n}$ we have $$\sigma (zq)=zqJ_nq^tz^t=zJ_nz^t=\sigma(z).$$
This shows that for any skew-symmetric matrix $A\in\cn^{2n\times 2n}$ the function $\hat\phi_A :\SU{2n}\to\cn$, given by
\begin{equation*}
\hat \phi_A(z) = \trace(AzJ_nz^t),
\end{equation*}
is $\Sp n$-invariant and hence induces a map $\phi_A:\SU{2n}/\Sp n\to\cn$ on the quotient space. Note that in this case the skew-symmetry condition can be assumed without loss of generality, since $\sigma(z)$ is skew-symmetric.
\smallskip

The following useful result can be found in \cite{Gud-Sif-Sob-2}.

\begin{proposition}
{\rm \cite{Gud-Sif-Sob-2}}\label{proposition-SUSp}
For $n\ge 2$, let $a,b$ be two linearly independent elements of $\cn^{2n}$ and $A\in\cn^{2n\times 2n}$ be the skew-symmetric matrix $$A=\sum_{j,\alpha=1}^{2n} a_jb_\alpha Y_{j\alpha}.$$ 
Further let $\phi_{a,b}:\SU{2n}/\Sp{n}\to\cn$ be the complex-valued function induced by its $\Sp n$-invariant lift $\hat\phi_{a,b}:\SU{2n}\to\cn$ given by
$$\hat\phi_{a,b}(z)=\trace(A z J_n z^t).$$
Then $\phi_{a,b}$ is an eigenfunction on $\SU{2n}/\Sp{n}$ with eigenvalues
$$\lambda=-\frac{2(2n^2-n-1)}{n}\ \ \text{and}\ \ \mu=-\frac{2(n-1)}{n}.$$
\end{proposition}
\smallskip

\begin{proof}(Theorem \ref{theorem-SUSp})
Let us assume that $z\in\SU{2n}$ is an element of the fibre $\hat\phi_{a,b}^{-1}(\{0\})$ i.e. $\hat\phi_{a,b}(z)=0$. The statement of Corollary \ref{corollary-phi0} guarantees the existence of such a point $z$.  Further assume that $z$ is a critical point of $\hat\phi_{a,b}:\SU{2n}\to\cn$ which means that $Y(\hat\phi_{a,b}(z))=0$ for every element in the orthogonal complement $\p$ of $\sp n$ in $\su{2n}$.
It can now be shown that for all $1\leq r<s\leq n,$ 
$$
Y_{rs}^{a}(\hat\phi_{a,b})=0, \ Y_{rs}^{b}(\hat\phi_{a,b})=0,\ X_{rs}^{a}(\hat\phi_{a,b})=0, \ Y_{rs}^{c}(\hat\phi_{a,b})=0 \ 
$$
if and only if 
\begin{eqnarray*}
0&=&(a,z_{n+s})(b,z_{r})-(a,z_{n+r})(b,z_{s})+(a,z_{s})(b,z_{n+r})-(a,z_{r})(b,z_{n+s}),\\
0&=&(a,z_{s})(b,z_{r})-(a,z_{r})(b,z_{s})-(a,z_{n+s})(b,z_{n+r})+(a,z_{n+r})(b,z_{n+s}),\\
0&=&(a,z_{n+s})(b,z_{r})+(a,z_{n+r})(b,z_{s})-(a,z_{s})(b,z_{n+r})-(a,z_{r})(b,z_{n+s}),\\
0&=&(a,z_{s})(b,z_{r})-(a,z_{r})(b,z_{s})+(a,z_{n+s})(b,z_{n+r})-(a,z_{n+r})(b,z_{n+s}).
\end{eqnarray*}
From the first and the third equations we immediately see that for all $1\leq r<s\leq n,$ we have
\begin{eqnarray}
(a,z_r)(b,z_{n+s})&=&(a,z_{n+s})(b,z_{r}),\label{suspeq1}\\
(a,z_{s})(b,z_{n+r})&=&(a,z_{n+r})(b,z_{s}).\label{suspeq2}
\end{eqnarray}  
Similarly, the second and the fourth equations tell us that for all $1\leq r<s\leq n$,
\begin{eqnarray}
(a,z_r)(b,z_{s})&=&(a,z_{s})(b,z_{r}),\label{suspeq3}\\
(a,z_{n+r})(b,z_{n+s})&=&(a,z_{n+s})(b,z_{n+r}).\label{suspeq4}
\end{eqnarray}

Since the columns of $\Bar{z}$ form a basis of $\cn^{2n}$, there exists an index $1\leq t\leq 2n$ such that $$\langle a,\Bar{z}_t\rangle =(a,z_t)\neq0.$$ Without loss of generality, we can assume that $1\leq t\leq n.$ Equations (\ref{suspeq1}) and \ref{suspeq3} now show that for all $1\leq s\leq n, \ s\neq t,$ it holds that $$(b,z_{n+s})=\frac{(a,z_{n+s})}{(a,z_t)}\cdot(b,z_t) \ \ \text{and} \ \ (b,z_{s})=\frac{(a,z_s)}{(a,z_t)}\cdot(b,z_t).$$
Now note that
$\phi_{a,b}(z)=0$ if and only if $$0=(a,z_{n+1})(b,z_{1})+\dots  +(a,z_{2n})(b,z_{n})-(a,z_{1})(b,z_{n+1})-\dots-(a,z_{n})(b,z_{2n}).$$
Entering our latest findings into this equation, we note that
\begin{eqnarray*}
    0&=&(a,z_{n+t})(b,z_t)+\sum_{s\neq t}(a,z_{n+s})\cdot\frac{(a,z_s)}{(a,z_t)}\cdot(b,z_t)\\
    &&-(a,z_t)(b,z_{n+t})-\sum_{s\neq t}(a,z_s)\cdot\frac{(a,z_{n+s})}{(a,z_t)}\cdot(b,z_t)\\
    &=&(a,z_{n+t})(b,z_{t})-(a,z_t)(b,z_{n+t}).
\end{eqnarray*}
Thus, we also have that $$(b,z_{n+t})=\frac{(a,z_{n+t})}{(a,z_t)}\cdot (b,z_{t}).$$
To summarise, we have now established that for every $1\leq s\leq 2n,$ $$\langle b,\Bar{z}_s\rangle =(b,z_{s})=\frac{(a,z_{s})}{(a,z_t)}\cdot(b,z_t)=\frac{(b,z_t)}{(a,z_t)}\cdot\langle a, \Bar{z}_s\rangle.$$ 
Thus, $$b=\sum_{s=1}^{2n}\langle b,\Bar{z}_s\rangle \Bar{z}_s=\frac{(b,z_t)}{(a,z_t)}\cdot\sum_{s=1}^{2n}\langle a,\Bar{z}_s\rangle \Bar{z}_s=\frac{(b,z_t)}{(a,z_t)}\cdot a.$$ 
This is a contradiction, since we have assumed that the elements $a$ and $b$ are linearly independent. We conclude that $\phi_{a,b}$ is regular over $0.$ 
By Theorem \ref{theorem-Gud-Mun-1}, $\phi_{a,b}^{-1}(\{0\})$ is a minimal submanifold of $\SU{2n}/\Sp{n}.$
\end{proof}

%%%%%%%%%%%%%%%%%%%%%%%%%%%%%%%%%%%%%%%%%%%%%%%%%%%%%%%%
\section{Acknowledgements}
%%%%%%%%%%%%%%%%%%%%%%%%%%%%%%%%%%%%%%%%%%%%%%%%%%%%%%%%

The authors are grateful to Thomas Jack Munn and Oskar Riedler for useful discussions on this work. 
They would also like to thank the anonymous referee for important suggestions on the presentation.

\appendix

%%%%%%%%%%%%%%%%%%%%%%%%%%%%%%%%%%%%%%%%%%%
\section{The linear system for $\SO{2n}/\U n$}
\label{appendix-A}
%%%%%%%%%%%%%%%%%%%%%%%%%%%%%%%%%%%%%%%%%%%

\medskip

\centerline{\bf The Case $n=2$.}
\smallskip

Here we assume that $(u,x_1)\neq 0$ i.e. $r=1$ and  $s=2$. Then obtain the following system of $4=2\,(2-1)+2$ equations
$$
\begin{bmatrix}
-(u,x_4) & (u,x_3) & -(u,x_2) &  (u,x_1) \\
-(u,x_2) & (u,x_1) &  (u,x_4) & -(u,x_3) \\
 (u,x_1) & (u,x_2) &  (u,x_3) &  (u,x_4) \\
 (u,x_3) & (u,x_4) & -(u,x_1) & -(u,x_2) \\
\end{bmatrix}
\cdot
\begin{bmatrix}
(b,x_1)\\(b,x_2)\\(b,x_3)\\(b,x_4)
\end{bmatrix}
=0.
$$
The matrix $M_2$ can then be transformed into $\tilde M_2=(u,x_1)\cdot I_4+S_2$.
$$
\tilde M_2=
\begin{bmatrix}
 (u,x_1) & (u,x_3) & -(u,x_4) &  (u,x_2) \\
-(u,x_3) & (u,x_1) & -(u,x_2) & -(u,x_4) \\
 (u,x_4) & (u,x_2) &  (u,x_1) & -(u,x_3) \\
-(u,x_2) & (u,x_4) &  (u,x_3) &  (u,x_1) \\
\end{bmatrix}
$$
It can be shown that  $|\det M_2|=|\det \tilde M_2|=|u|^4\neq 0$.

\medskip\medskip
\centerline{\bf The Case $n=3$.}
\smallskip

Here we assume that $(u,x_1)\neq 0$ i.e. $r=1$ and  $s=2,3$ and yield the following matrix $M_3\in\rn^{6\times 6}$
$$M_3=
\begin{bmatrix}
-(u,x_5) & (u,x_4) &       0 & -(u,x_2) &  (u,x_1) & 0 \\
-(u,x_2) & (u,x_1) &       0 &  (u,x_5) & -(u,x_4) & 0 \\
-(u,x_6) &       0 & (u,x_4) & -(u,x_3) &        0 &  (u,x_1) \\
-(u,x_3) &       0 & (u,x_1) &  (u,x_6) &        0 & -(u,x_4) \\
 (u,x_1) & (u,x_2) & (u,x_3) &  (u,x_4) &  (u,x_5) &  (u,x_6) \\
 (u,x_4) & (u,x_5) & (u,x_6) & -(u,x_1) & -(u,x_2) & -(u,x_3) \\
\end{bmatrix}
.
$$
The matrix $M_3$ can then be transformed into $\tilde M_3=(u,x_1)\cdot I_6+S_3$.
$$\tilde M_3=
\begin{bmatrix}
 (u,x_1)& (u,x_4)&        0&       0 & -(u,x_5) &  (u,x_2)\\
-(u,x_4)& (u,x_1)&        0&       0 & -(u,x_2) & -(u,x_5)\\
       0&       0&  (u,x_1)& (u,x_4) & -(u,x_6) &  (u,x_3)\\
       0&       0& -(u,x_4)& (u,x_1) & -(u,x_3) & -(u,x_6)\\
 (u,x_5)& (u,x_2)&  (u,x_6)& (u,x_3) &  (u,x_1) & -(u,x_4)\\
-(u,x_2)& (u,x_5)& -(u,x_3)& (u,x_6) &  (u,x_4) &  (u,x_1)\\
\end{bmatrix}
.
$$
$$|\det M_3|=|\det \tilde M_3|=(\,(u,x_1)^2+(u,x_4)^2\,)\cdot |u|^4\neq 0.$$

\medskip\medskip
\centerline{\bf The General Case $n\ge 2$.}
\smallskip

Here we assume that $(u,x_1)\neq 0$ i.e. $r=1$ and  $s=2,3,\dots n$ and then obtain the following system of $2n=2\,(n-1)+2$ equations
\begin{eqnarray*}
0&=&(u,x_r)(b,x_{n+s})+(u,x_{n+r})(b,x_s)-(u,x_s)(b,x_{n+r})-(u,x_{n+s})(b,x_r),\\
0&=&(u,x_r)(b,x_{s})+(u,x_{n+s})(b,x_{n+r})-(u,x_s)(b,x_{r})-(u,x_{n+r})(b,x_{n+s}),\\
0&=&\Re\, (a,b)=\sum_{k=1}^{2n}(u,x_k)\,(b,x_k),\\
0&=&\Re\, \hat\phi_{a,b}(x)=-\sum_{k=1}^{n}
\big( (u,x_{n+k})\,(b,x_k)-(u,x_{k})\,(b,x_{n+k})\big).
\end{eqnarray*}

The corresponding matrix $M_{n}$ can then be transformed into $$\tilde M_{n}=(u,x_1)\cdot I_{2n}+S_{n}.$$

%%%%%%%%%%%%%%%%%%%%%%%%%%%%%%%%%%%%%%%%%%%


\begin{thebibliography}{99}
%%%%%%%%%%%%%%%%%%%%%%%%%%%%%%%%%%%%%%%%%%%


\bibitem{Bai-Eel}
P. Baird, J. Eells,
{\it A conservation law for harmonic maps},
Geometry Symposium Utrecht 1980, Lecture Notes in Mathematics {\bf 894}, Springer (1981), 1-25.


\bibitem{Bai-Woo-book}
P. Baird and J. C. Wood,
{\it Harmonic morphisms between Riemannian manifolds},
London Math. Soc. Monogr. {\bf 29},
Oxford Univ. Press (2003).


\bibitem{Bur-Raw}
F. E. Burstall, J. H. Rawnsley,
{\it Twistor theory for Riemannian symmetric spaces.
With applications to harmonic maps of Riemann surfaces}
Lecture Notes in Math. {\bf 1424}, Springer (1990)


\bibitem{Bur-Gue}
F.E. Burstall, M.A. Guest, 
{\it Harmonic two-spheres in compact symmetric spaces, revisited},
Math. Ann. {\bf 309} (1997), 541-752.


\bibitem{Cal}
E. Calabi, 
{\it Minimal immersions of surfaces in Euclidean spheres}, 
J. Differential Geom. {\bf 1} (1967), 111-125.


\bibitem{Eel-Sam}
J. Eells, J. H. Sampson,
{\it Harmonic mappings of Riemannian manifolds},
Amer. J. Math. {\bf 86} (1964), 109-160.


\bibitem{Eel-Woo}
J. Eells, J.C. Wood, 
{\it Harmonic maps from surfaces to complex projective spaces}, 
Adv. in Math. {\bf 49}, (1983), 217-263.


\bibitem{Geg-MSc}
J. M. Gegenfurtner,
{\it Minimal Submanifolds of the Classical Compact Riemannian Symmetric Spaces},
Master's dissertation, University of Lund (2024).


{\tt www.matematik.lu.se/matematiklu/personal/sigma/students/ Johanna-Marie-Gegenfurtner-MSc.pdf}


\bibitem{Gha-Gud-4}
E. Ghandour, S. Gudmundsson,
{\it Explicit $p$-harmonic functions on the real Grassmannians}, 
Adv. Geom. {\bf 23} (2023), 315–321.


\bibitem{Gha-Gud-5}
E. Ghandour, S. Gudmundsson,
{\it Explicit harmonic morphisms and $p$-harmonic functions from the complex and quaternionic Grassmannians}, 
Ann. Global Anal. Geom. {\bf 64} (2023), 15, 18 pp.


\bibitem{Gud-bib}
S. Gudmundsson,
{\it The Bibliography of Harmonic Morphisms},
{\tt www.matematik.lu.se/ matematiklu/personal/sigma/harmonic/bibliography.html}


\bibitem{Gud-Mun-1}
S. Gudmundsson, T. J. Munn,
{\it Minimal submanifolds via complex-valued eigenfunctions},
J. Geom. Anal. {\bf 34} (2024), 190, 22 pp.


\bibitem{Gud-Sak-1}
S. Gudmundsson, A. Sakovich,
{\it Harmonic morphisms from the classical compact semisimple Lie groups},
Ann. Global Anal. Geom. {\bf 33} (2008), 343-356.


\bibitem{Gud-Sif-Sob-2}
S. Gudmundsson, A. Siffert, M. Sobak,
{\it Explicit proper $p$-harmonic functions on the Riemannian symmetric spaces $\SU n/\SO n$, $\Sp n/\U n$, $\SO {2n}/\U n$, $\SU{2n}/\Sp n$},
J. Geom. Anal. {\bf 32} (2022), 147, 16 pp.


\bibitem{Gud-Sob-1}
S. Gudmundsson, M. Sobak,
{\it Proper $r$-harmonic functions from Riemannian manifolds},
Ann. Global Anal. Geom. {\bf 57} (2020), 217-223.





\bibitem{Rie-Sif}
O. Riedler, A. Siffert,
{\it Global eigenfamilies on compact manifolds}, preprint, University of M\" unster (2024).


\bibitem{Uhl}
K. Uhlenbeck, 
{\it Harmonic maps into Lie groups: classical solutions of the chiral model},
J. Differential Geom. {\bf 30} (1989), 1–50.

\end{thebibliography}
\end{document}